\documentclass{amsart}
\usepackage{setspace}
\usepackage{a4}
\usepackage{amsthm}
\usepackage{latexsym}
\usepackage{amsfonts}
\usepackage{graphicx}
\usepackage{textcomp}
\usepackage{cite}
\usepackage{enumerate}
\usepackage{amssymb}
\usepackage{hyperref}
\usepackage{amsmath}
\usepackage{tikz}
\usepackage[mathscr]{euscript}
\usepackage{mathtools}
\newtheorem{theorem}{Theorem}[section]

\newtheorem{corollary}[theorem] {Corollary}

\newtheorem{problem}[theorem]{Problem}

\title{This is the title}
\usepackage{fancyhdr}

\begin{document}
\hrule\hrule\hrule\hrule\hrule
\vspace{0.3cm}	
\begin{center}
{\bf\large{{Non-Archimedean Cauchy-Schwarz Angle-Length and Chebyshev Arithmetic Mean Inequalities}}}\\
\vspace{0.3cm}
\hrule\hrule\hrule\hrule\hrule
\vspace{0.3cm}
\textbf{K. Mahesh Krishna}\\
School of Mathematics and Natural Sciences\\
Chanakya University Global Campus\\
NH-648, Haraluru Village\\
Devanahalli Taluk, 	Bengaluru  North District\\
Karnataka State, 562 110, India\\
Email: kmaheshak@gmail.com\\

Date: \today
\end{center}

\hrule\hrule
\vspace{0.5cm}
%--------------------------------------
\textbf{Abstract}: Let $\mathbb{K}$ be a non-Archimedean valued field. Let $n \in \mathbb{N}$. For every  $(a_j)_{j=1}^n, (b_j)_{j=1}^n \in \mathbb{K}^n$, we show that 
\begin{align*}
		\left|\sum_{j=1}^{n}a_jb_j\right|^2\leq \max\left\{\left|\sum_{j=1}^n a_j^2\right|\left|\sum_{k=1}^nb^2_k\right|, \max_{1\leq j<k \leq n}|a_jb_k-a_kb_j|^2\right\},
\end{align*}
\begin{align*}
	\left|\sum_{j=1}^n a_j^2\right|\left|\sum_{k=1}^nb^2_k\right|\leq \max\left\{\left|\sum_{j=1}^n a_jb_j\right|^2, \max_{1\leq j<k \leq n}|a_jb_k-a_kb_j|^2\right\},
\end{align*}	
\begin{align*}
	\left|AM(a_j)_{j=1}^n\right|\left|AM(b_j)_{j=1}^n\right|\leq \max\left\{	\left|AM(a_jb_j)_{j=1}^n\right|, \frac{1}{|n|^2}\max_{1\leq j < k \leq n}|a_j-a_k||b_j-b_k|\right\},
\end{align*}
	\begin{align*}
	\left|AM(a_jb_j)_{j=1}^n\right|\leq \max\left\{	\left|AM(a_j)_{j=1}^n\right|\left|AM(b_j)_{j=1}^n\right|, \frac{1}{|n|^2}\max_{1\leq j < k \leq n}|a_j-a_k||b_j-b_k|\right\}, 
\end{align*}
where $AM$ denotes the arithmetic mean. First and second are non-Archimedean versions of Cauchy-Schwarz angle-length and third and fourth are Chebyshev arithmetic mean inequalities. Unlike in the Archimedean case, no conditions are required to derive non-Archimedean Chebyshev arithmetic mean inequalities.



\hrule

%\tableofcontents
\hrule
\section{Introduction}
Universally known Cauchy-Schwarz angle-length inequality, derived in 1821,  is the following. 
\begin{theorem}\cite{STEELE, MITRINOVIC, MITRINOVICPECARICFINK, HARDYLITTLEWOODPOLYA, VENKATACHALA, ENGEL, BECKENBACHBELLMAN}
	(\textbf{Cauchy-Schwarz Angle-Length Inequality}) \label{CSI}
Let $n \in \mathbb{N}$. Then 
\begin{align*}
	\left(\sum_{j=1}^{n}a_jb_j\right)^2\leq \left(\sum_{j=1}^{n}a_j^2\right)\left(\sum_{j=1}^{n}b_j^2\right), \quad \forall (a_j)_{j=1}^n, (b_j)_{j=1}^n \in \mathbb{R}^n.
\end{align*}
\end{theorem}
\begin{theorem}\cite{STEELE, MITRINOVIC, MITRINOVICPECARICFINK, HARDYLITTLEWOODPOLYA, VENKATACHALA, ENGEL, BECKENBACHBELLMAN} (\textbf{Cauchy-Schwarz Angle-Length Inequality}) Let $\{a_j\}_{j=1}^\infty, \{b_j\}_{j=1}^\infty \subseteq \mathbb{R}$ be such that 
		\begin{align*}
		\sum_{j=1}^{\infty}a^2_j\in \mathbb{R} \quad \text{and} 	\quad \sum_{j=1}^{\infty}b^2_j\in \mathbb{R}.
	\end{align*}
Then 
\begin{align*}
	\left(\sum_{j=1}^{\infty}a_jb_j\right)^2\leq \left(\sum_{j=1}^{\infty}a_j^2\right)\left(\sum_{j=1}^{\infty}b_j^2\right).
\end{align*}
\end{theorem}
Let $n \in \mathbb{N}$. Let $(a_j)_{j=1}^n, (b_j)_{j=1}^n \in \mathbb{R}^n$. Recall that the arithmetic means of $(a_j)_{j=1}^n, (b_j)_{j=1}^n$ and $(a_jb_j)_{j=1}^n$ are given by 
\begin{align*}
	AM(a_j)_{j=1}^n\coloneqq \frac{\sum_{j=1}^{n}a_j}{n}, \quad 	AM(a_j)_{j=1}^n\coloneqq \frac{\sum_{j=1}^{n}b_j}{n}, \quad 	AM(a_jb_j)_{j=1}^n\coloneqq \frac{\sum_{j=1}^{n}a_jb_j}{n}. 
\end{align*} 
Chebyshev arithmetic mean inequalities, derived in 1882, gives a relation between $AM(a_j)_{j=1}^n$, $AM(b_j)_{j=1}^n$ and $AM(a_jb_j)_{j=1}^n$ whenever $a_1\leq \cdots \leq a_n$ (resp. $a_1\geq \cdots \geq a_n$) and $ b_1\leq \cdots \leq b_n$ (resp. $ b_1\geq \cdots \geq b_n$).  
\begin{theorem} \cite{MITRINOVICVASIC, FARHADIANPONOMARENKO, VENKATACHALA, ENGEL, HARDYLITTLEWOODPOLYA, BESENYEI}
\label{CI} (\textbf{Chebyshev Arithmetic Mean Inequality})
Let $n \in \mathbb{N}$. Let $(a_j)_{j=1}^n, (b_j)_{j=1}^n \in \mathbb{R}^n$ be such that  
\begin{align*}
	a_1\leq \cdots \leq a_n \quad (\text{resp. }a_1\geq \cdots \geq a_n)
\end{align*} 
and 
\begin{align*}
	b_1\leq \cdots \leq b_n\quad (\text{resp. }b_1\geq \cdots \geq b_n).
\end{align*} 
Then 
\begin{align*}
\left(\sum_{j=1}^{n}a_j\right)\left(\sum_{j=1}^{n}b_j\right)\leq n\left(\sum_{j=1}^{n}a_jb_j\right).
\end{align*}
In other words, 
\begin{align*}
	\left(AM(a_j)_{j=1}^n\right)\left(AM(b_j)_{j=1}^n\right)\leq AM(a_jb_j)_{j=1}^n.
\end{align*}
\end{theorem}
Theorem \ref{CI} has following extension. 
\begin{theorem} \cite{MITRINOVICVASIC, FARHADIANPONOMARENKO, VENKATACHALA, ENGEL, HARDYLITTLEWOODPOLYA, BESENYEI, WRIGHT}
	\label{CI2} (\textbf{Chebyshev Arithmetic Mean Inequality})
	Let $n \in \mathbb{N}$. Let $ r_1, \dots, r_n\in [0, \infty)$ be such that $\sum_{j=1}^{n}r_j=1$.
	Let $(a_j)_{j=1}^n, (b_j)_{j=1}^n \in \mathbb{R}^n$ be such that  
	\begin{align*}
		a_1\leq \cdots \leq a_n \quad (\text{resp. }a_1\geq \cdots \geq a_n)
	\end{align*} 
	and 
	\begin{align*}
		b_1\leq \cdots \leq b_n\quad (\text{resp. }b_1\geq \cdots \geq b_n).
	\end{align*} 
	Then 
	\begin{align*}
		\left(\sum_{j=1}^{n}r_ja_j\right)\left(\sum_{k=1}^{n}r_kb_k\right)\leq \sum_{j=1}^{n}r_ja_jb_j.
	\end{align*}
\end{theorem}
\begin{corollary} \cite{MITRINOVICVASIC, FARHADIANPONOMARENKO, VENKATACHALA, ENGEL, HARDYLITTLEWOODPOLYA, BESENYEI}
	 (\textbf{Chebyshev Arithmetic Mean Inequality})
	Let $n \in \mathbb{N}$. Let $ r_1, \dots, r_n\in [0, \infty)$. 
	Let $(a_j)_{j=1}^n, (b_j)_{j=1}^n \in \mathbb{R}^n$ be such that  
	\begin{align*}
		a_1\leq \cdots \leq a_n \quad (\text{resp. }a_1\geq \cdots \geq a_n)
	\end{align*} 
	and 
	\begin{align*}
		b_1\leq \cdots \leq b_n\quad (\text{resp. }b_1\geq \cdots \geq b_n).
	\end{align*} 
	Then 
	\begin{align*}
		\left(\sum_{j=1}^{n}r_ja_j\right)\left(\sum_{k=1}^{n}r_kb_k\right)\leq \left(\sum_{k=1}^{n}r_k\right)\left(\sum_{j=1}^{n}r_ja_jb_j\right).
	\end{align*}
\end{corollary}
Theorem \ref{CI2} easily extends to infinite sequences. 
\begin{theorem}\cite{MITRINOVICVASIC, FARHADIANPONOMARENKO, VENKATACHALA, ENGEL, HARDYLITTLEWOODPOLYA, BESENYEI}
 (\textbf{Chebyshev Arithmetic Mean Inequality})
 Let $ \{r_j\}_{j=1}^\infty\subseteq [0, \infty)$ be such that $\sum_{j=1}^{\infty}r_j=1$. Let $\{a_j\}_{j=1}^\infty, \{b_j\}_{j=1}^\infty \subseteq \mathbb{R}$ be such that  
\begin{align*}
	a_1\leq a_2 \leq  \cdots  \quad (\text{resp. }a_1\geq a_2 \geq \cdots )
\end{align*} 
and 
\begin{align*}
	b_1\leq b_2 \leq \cdots \quad (\text{resp. }b_1\geq b_2 \geq \cdots ).
\end{align*} 
Assume that 
\begin{align*}
	\sum_{j=1}^{\infty}a_j\in \mathbb{R} \quad \text{and} 	\quad \sum_{j=1}^{\infty}b_j\in \mathbb{R}.
\end{align*}
Then 
\begin{align*}
	\left(\sum_{j=1}^{\infty}r_ja_j\right)\left(\sum_{k=1}^{\infty}r_kb_k\right)\leq \sum_{j=1}^{\infty}r_ja_jb_j.
\end{align*}	
\end{theorem}
\begin{corollary}
	\cite{MITRINOVICVASIC, FARHADIANPONOMARENKO, VENKATACHALA, ENGEL, HARDYLITTLEWOODPOLYA, BESENYEI}
	(\textbf{Chebyshev Arithmetic Mean Inequality})
	Let $ \{r_j\}_{j=1}^\infty\subseteq [0, \infty)$ be such that $\sum_{j=1}^{\infty}r_j\in \mathbb{R}$. Let $\{a_j\}_{j=1}^\infty, \{b_j\}_{j=1}^\infty \subseteq \mathbb{R}$ be such that  
	\begin{align*}
		a_1\leq a_2 \leq  \cdots  \quad (\text{resp. }a_1\geq a_2 \geq \cdots )
	\end{align*} 
	and 
	\begin{align*}
		b_1\leq b_2 \leq \cdots \quad (\text{resp. }b_1\geq b_2 \geq \cdots ).
	\end{align*} 
	Assume that 
	\begin{align*}
		\sum_{j=1}^{\infty}a_j\in \mathbb{R} \quad \text{and} 	\quad \sum_{j=1}^{\infty}b_j\in \mathbb{R}.
	\end{align*}
	Then 
	\begin{align*}
		\left(\sum_{j=1}^{\infty}r_ja_j\right)\left(\sum_{k=1}^{\infty}r_kb_k\right)\leq\left(\sum_{k=1}^{\infty}r_k\right) \left(\sum_{j=1}^{\infty}r_ja_jb_j\right).
	\end{align*}	
\end{corollary}
Following is the Chebyshev arithmetic mean inequality whenever sequences are ordered in different ways. 
\begin{theorem}	\cite{MITRINOVICVASIC, FARHADIANPONOMARENKO, VENKATACHALA, ENGEL, HARDYLITTLEWOODPOLYA, BESENYEI} \label{CR}
	(\textbf{Chebyshev Arithmetic Mean Inequality})
	Let $ \{r_j\}_{j=1}^\infty\subseteq [0, \infty)$ be such that $\sum_{j=1}^{\infty}r_j\in \mathbb{R}$. Let $\{a_j\}_{j=1}^\infty, \{b_j\}_{j=1}^\infty \subseteq \mathbb{R}$ be such that  
	\begin{align*}
		a_1\leq a_2 \leq  \cdots  \quad (\text{resp. }a_1\geq a_2 \geq \cdots )
	\end{align*} 
	and 
	\begin{align*}
		b_1\geq b_2 \geq \cdots \quad (\text{resp. }b_1\leq b_2 \leq \cdots ).
	\end{align*} 
	Assume that 
	\begin{align*}
		\sum_{j=1}^{\infty}a_j\in \mathbb{R} \quad \text{and} 	\quad \sum_{j=1}^{\infty}b_j\in \mathbb{R}.
	\end{align*}
	Then 
	\begin{align*}
		\left(\sum_{j=1}^{\infty}r_ja_j\right)\left(\sum_{k=1}^{\infty}r_kb_k\right)\geq\left(\sum_{k=1}^{\infty}r_k\right) \left(\sum_{j=1}^{\infty}r_ja_jb_j\right).
	\end{align*}	
	
\end{theorem}

It is natural and important to ask what is non-Archimedean version of Theorem \ref{CSI}, Theorem \ref{CI},  Theorem \ref{CI2} and Theorem \ref{CR}? We answer the question.

\section{Non-Archimedean Cauchy-Schwarz Angle-Length and Chebyshev Arithmetic Mean Inequalities}
Let $\mathbb{K}$ be a field. Recall that a map $|\cdot|: \mathbb{K} \to [0, \infty)$ is said to be a non-Archimedean  valuation  if following conditions holds.
\begin{enumerate}[\upshape(i)]
	\item If $\lambda  \in \mathbb{K}$ is such that $|\lambda|=0$, then $\lambda=0$.
	\item $|\lambda \mu|=|\lambda||\mu|$ for all $\lambda, \mu  \in \mathbb{K}$.
	\item (Ultra-triangle inequality) $|\lambda+\mu|\leq \max\{|\lambda|, |\mu|\}$ for all $\lambda, \mu \in \mathbb{K}$.
\end{enumerate}
In this case, $\mathbb{K}$ is called as non-Archimedean valued field \cite{SCHIKHOF}.  We first derive non-Archimedean Cauchy-Schwarz angle-length inequality.
\begin{theorem}\label{NACS} (\textbf{Non-Archimedean Cauchy-Schwarz Angle-Length Inequality}) Let $n \in \mathbb{N}$ and $\mathbb{K}$ be a non-Archimedean valued field. Let $(a_j)_{j=1}^n, (b_j)_{j=1}^n \in \mathbb{K}^n$. Then 
\begin{align*}
	\left|\sum_{j=1}^{n}a_jb_j\right|^2&\leq \max\left\{\left|\sum_{j=1}^n a_j^2\right|\left|\sum_{k=1}^nb^2_k\right|, \max_{1\leq j<k \leq n}|a_jb_k-a_kb_j|^2\right\}\\
	&\leq \left(\max_{1\leq j \leq n}|a_j|\right)^2\left(\max_{1\leq k \leq n}|b_k|\right)^2.
\end{align*}	
\end{theorem}
\begin{proof}
Our proof is motivated from the arguments for the Archimedean case by Cauchy using Lagrange identity  \cite{STEELE, MALIGRANDA, MITRINOVIC, MITRINOVICPECARICFINK, GIDEANICULESCU, GUASTI}. The Lagrange identity says that 
\begin{align*}
	\sum_{1\leq j<k \leq n}(a_jb_k-a_kb_j)^2&=\frac{1}{2}\sum_{j=1}^{n}\sum_{k=1}^{n}(a_jb_k-a_kb_j)^2\\
	&=\left(\sum_{j=1}^n a_j^2\right)\left(\sum_{k=1}^n b_k^2\right)-\left(\sum_{j=1}^n a_jb_j\right)^2.
\end{align*}
Rearranging, we get 
\begin{align*}
	\left(\sum_{j=1}^n a_jb_j\right)^2&=\left(\sum_{j=1}^n a_j^2\right)\left(\sum_{k=1}^n b_k^2\right)-\sum_{1\leq j<k \leq n}(a_jb_k-a_kb_j)^2.
\end{align*}
By taking non-Archimedean absolute value, we get 
\begin{align*}
	\left|\sum_{j=1}^n a_jb_j\right|^2&=\left|\left(\sum_{j=1}^n a_j^2\right)\left(\sum_{k=1}^n b_k^2\right)-\sum_{1\leq j<k \leq n}(a_jb_k-a_kb_j)^2\right|\\
	&\leq \max\left\{\left|\sum_{j=1}^n a_j^2\right|\left|\sum_{k=1}^nb^2_k\right|, \left|\sum_{1\leq j<k \leq n}(a_jb_k-a_kb_j)^2\right|\right\}\\
	&\leq \max\left\{\left|\sum_{j=1}^n a_j^2\right|\left|\sum_{k=1}^nb^2_k\right|, \max_{1\leq j<k \leq n}|a_jb_k-a_kb_j|^2\right\}.
\end{align*}
\end{proof}
\begin{theorem}
(\textbf{Non-Archimedean Cauchy-Schwarz Angle-Length Inequality}) Let $n \in \mathbb{N}$ and $\mathbb{K}$ be a non-Archimedean valued field. Let $(a_j)_{j=1}^n, (b_j)_{j=1}^n \in \mathbb{K}^n$. Then 
\begin{align*}
	\left|\sum_{j=1}^n a_j^2\right|\left|\sum_{k=1}^nb^2_k\right|\leq \max\left\{\left|\sum_{j=1}^n a_jb_j\right|^2, \max_{1\leq j<k \leq n}|a_jb_k-a_kb_j|^2\right\}.
\end{align*}		
\end{theorem}
\begin{proof}
	From the computations done in the proof of Theorem \ref{NACS}, we get
	\begin{align*}
		\left(\sum_{j=1}^n a_j^2\right)\left(\sum_{k=1}^n b_k^2\right)=	\left(\sum_{j=1}^n a_jb_j\right)^2+\sum_{1\leq j<k \leq n}(a_jb_k-a_kb_j)^2.
	\end{align*}
	By taking non-Archimedean absolute value, we get 
	\begin{align*}
	\left|\sum_{j=1}^n a_j^2\right|\left|\sum_{k=1}^nb^2_k\right|&=\left|\left(\sum_{j=1}^n a_jb_j\right)^2+\sum_{1\leq j<k \leq n}(a_jb_k-a_kb_j)^2\right|\\
	&\leq \max\left\{\left|\sum_{j=1}^n a_jb_j\right|^2, \left|\sum_{1\leq j<k \leq n}(a_jb_k-a_kb_j)^2\right|\right\}\\
	&\leq \max\left\{\left|\sum_{j=1}^n a_jb_j\right|^2, \max_{1\leq j<k \leq n}|a_jb_k-a_kb_j|^2\right\}.
	\end{align*}
\end{proof}
Previous theorems easily extend to infinite sequences. 
\begin{theorem}(\textbf{Non-Archimedean Cauchy-Schwarz Angle-Length Inequality}) Let $\mathbb{K}$ be a non-Archimedean valued field.
Let $\{a_j\}_{j=1}^\infty, \{b_j\}_{j=1}^\infty  \subseteq \mathbb{K}$. Assume that 
\begin{align*}
	\lim\limits_{j \to \infty }a_j=0 \quad \text{and} \quad 	\lim\limits_{j \to \infty }b_j=0.
\end{align*}
Then 
\begin{align*}
	\left|\sum_{j=1}^{\infty}a_jb_j\right|^2\leq \max\left\{\left|\sum_{j=1}^\infty a_j^2\right|\left|\sum_{k=1}^\infty b^2_k\right|, \max_{1\leq j<k <\infty}|a_jb_k-a_kb_j|^2\right\},
\end{align*}	
\end{theorem}
\begin{theorem}(\textbf{Non-Archimedean Cauchy-Schwarz Angle-Length Inequality}) Let $\mathbb{K}$ be a non-Archimedean valued field.
	Let $\{a_j\}_{j=1}^\infty, \{b_j\}_{j=1}^\infty  \subseteq \mathbb{K}$. Assume that 
	\begin{align*}
		\lim\limits_{j \to \infty }a_j=0 \quad \text{and} \quad 	\lim\limits_{j \to \infty }b_j=0.
	\end{align*}
	Then 
	\begin{align*}
		\left|\sum_{j=1}^\infty a_j^2\right|\left|\sum_{k=1}^\infty b^2_k\right|\leq \max\left\{\left|\sum_{j=1}^\infty a_jb_j\right|^2, \max_{1\leq j<k <\infty }|a_jb_k-a_kb_j|^2\right\}.
	\end{align*}	
\end{theorem}
We now derive non-Archimedean Chebyshev arithmetic mean inequalities.
\begin{theorem} (\textbf{Non-Archimedean Chebyshev Arithmetic Mean Inequality})
Let $n \in \mathbb{N}$ and $\mathbb{K}$ be a non-Archimedean valued field. Let $(a_j)_{j=1}^n, (b_j)_{j=1}^n \in \mathbb{K}^n$. Then 
\begin{align*}
		\left|\sum_{j=1}^{n}a_j\right|\left|\sum_{k=1}^{n}b_k\right|\leq 	 \max\left\{|n|\left|\sum_{j=1}^{n}a_jb_j\right|, \max_{1\leq j < k \leq n}|a_j-a_k||b_j-b_k|\right\}.
\end{align*}
In other words, 
\begin{align*}
		\left|AM(a_j)_{j=1}^n\right|\left|AM(b_j)_{j=1}^n\right|\leq \max\left\{	\left|AM(a_jb_j)_{j=1}^n\right|, \frac{1}{|n|^2}\max_{1\leq j < k \leq n}|a_j-a_k||b_j-b_k|\right\}.
\end{align*}
\end{theorem}
\begin{proof}
Our proof is motivated by the arguments for the Archimedean case by Korkin \cite{MITRINOVICVASIC}. We have 
\begin{align}\label{1}
	\sum_{j=1}^{n}\sum_{k=1}^{n}(a_jb_j-a_jb_k)=n\sum_{j=1}^{n}a_jb_j-\left(\sum_{j=1}^{n}a_j\right)\left(\sum_{k=1}^{n}b_k\right)
\end{align}
and 
\begin{align}\label{2}
	\sum_{j=1}^{n}\sum_{k=1}^{n}(a_kb_k-a_kb_j)=n\sum_{k=1}^{n}a_kb_k-\left(\sum_{k=1}^{n}a_k\right)\left(\sum_{j=1}^{n}b_j\right).
\end{align}
By adding (\ref{1}) and (\ref{2}) we get 
\begin{align*}
\sum_{j=1}^{n}\sum_{k=1}^{n}(a_j-a_k)(b_j-b_k)	&=\sum_{j=1}^{n}\sum_{k=1}^{n}[(a_jb_j-a_jb_k)+(a_kb_k-a_kb_j)]\\
&=2 \left[n\sum_{j=1}^{n}a_jb_j-\left(\sum_{j=1}^{n}a_j\right)\left(\sum_{k=1}^{n}b_k\right)\right].
\end{align*}
Rearranging, we get 
\begin{align*}
\left(\sum_{j=1}^{n}a_j\right)\left(\sum_{k=1}^{n}b_k\right)&=n\sum_{j=1}^{n}a_jb_j-\frac{1}{2}\sum_{j=1}^{n}\sum_{k=1}^{n}(a_j-a_k)(b_j-b_k)\\
&=n\sum_{j=1}^{n}a_jb_j-\sum_{1\leq j<k\leq n}(a_j-a_k)(b_j-b_k).
\end{align*}
By taking non-Archimedean absolute value, we get 
\begin{align*}
	\left|\sum_{j=1}^{n}a_j\right|\left|\sum_{k=1}^{n}b_k\right|&=\left|n\sum_{j=1}^{n}a_jb_j-\sum_{1\leq j<k\leq n}(a_j-a_k)(b_j-b_k)	\right|\\
	&\leq \max\left\{\left|n\sum_{j=1}^{n}a_jb_j\right|, \left|\sum_{1\leq j<k\leq n}(a_j-a_k)(b_j-b_k)	\right|\right\}\\
	&=\max\left\{|n|\left|\sum_{j=1}^{n}a_jb_j\right|, \left|\sum_{1\leq j<k\leq n}(a_j-a_k)(b_j-b_k)\right|\right\}\\
	&\leq  \max\left\{|n|\left|\sum_{j=1}^{n}a_jb_j\right|, \max_{1\leq j < k \leq n}|a_j-a_k||b_j-b_k|\right\}.
\end{align*}
Dividing by $|n|^2$ completes the proof. 
\end{proof}
\begin{theorem} (\textbf{Non-Archimedean Chebyshev Arithmetic Mean Inequality})\label{NACAMI}
	Let $n \in \mathbb{N}$ and $\mathbb{K}$ be a non-Archimedean valued field.  Let $ r_1, \dots, r_n\in \mathbb{K}$ be such that $\sum_{j=1}^{n}r_j=1$. Let $(a_j)_{j=1}^n, (b_j)_{j=1}^n \in \mathbb{K}^n$. Then 
	\begin{align*}
		\left|\sum_{j=1}^{n}r_ja_j\right|\left|\sum_{k=1}^{n}r_kb_k\right|\leq 	 \max\left\{\left|\sum_{j=1}^{n}r_ja_jb_j\right|, \max_{1\leq j < k \leq n}|r_j||r_k||a_j-a_k||b_j-b_k|\right\}.
	\end{align*}
\end{theorem}
\begin{proof}
	We have 
	\begin{align}\label{3}
		\sum_{j=1}^{n}\sum_{k=1}^{n}r_jr_k(a_jb_j-a_jb_k)=\sum_{j=1}^{n}r_ja_jb_j-\left(\sum_{j=1}^{n}r_ja_j\right)\left(\sum_{k=1}^{n}r_kb_k\right)
	\end{align}
	and 
	\begin{align}\label{4}
		\sum_{j=1}^{n}\sum_{k=1}^{n}r_kr_j(a_kb_k-a_kb_j)=\sum_{k=1}^{n}r_ka_kb_k-\left(\sum_{k=1}^{n}r_ka_k\right)\left(\sum_{j=1}^{n}r_jb_j\right).
	\end{align}
	By adding (\ref{3}) and (\ref{4}) we get 
 	\begin{align*}
		\sum_{j=1}^{n}\sum_{k=1}^{n}r_jr_k(a_j-a_k)(b_j-b_k)	&=\sum_{j=1}^{n}\sum_{k=1}^{n}r_jr_k[(a_jb_j-a_jb_k)+(a_kb_k-a_kb_j)]\\
		&=2 \left[\sum_{j=1}^{n}r_ja_jb_j-\left(\sum_{j=1}^{n}r_ja_j\right)\left(\sum_{k=1}^{n}r_kb_k\right)\right].
	\end{align*}
	Rearranging, we get
	\begin{align*}
		\left(\sum_{j=1}^{n}r_ja_j\right)\left(\sum_{k=1}^{n}r_kb_k\right)&=\sum_{j=1}^{n}r_ja_jb_j-\frac{1}{2}\sum_{j=1}^{n}\sum_{k=1}^{n}r_jr_k(a_j-a_k)(b_j-b_k)\\
		&=\sum_{j=1}^{n}r_ja_jb_j-\sum_{1\leq j<k\leq n}r_jr_k(a_j-a_k)(b_j-b_k).
	\end{align*}
	By taking non-Archimedean absolute value, we get 
	\begin{align*}
		\left|\sum_{j=1}^{n}r_ja_j\right|\left|\sum_{k=1}^{n}r_kb_k\right|&=\left|\sum_{j=1}^{n}r_ja_jb_j-\sum_{1\leq j<k\leq n}r_jr_k(a_j-a_k)(b_j-b_k)	\right|\\
		&\leq \max\left\{\left|\sum_{j=1}^{n}r_ja_jb_j\right|, \left|\sum_{1\leq j<k\leq n}r_jr_k(a_j-a_k)(b_j-b_k)	\right|\right\}\\
		&=\max\left\{\left|\sum_{j=1}^{n}r_ja_jb_j\right|, \left|\sum_{1\leq j<k\leq n}r_jr_k(a_j-a_k)(b_j-b_k)\right|\right\}\\
		&\leq  \max\left\{\left|\sum_{j=1}^{n}r_ja_jb_j\right|, \max_{1\leq j < k \leq n}|r_j||r_k||a_j-a_k||b_j-b_k|\right\}.
	\end{align*}
\end{proof}
\begin{corollary}(\textbf{Non-Archimedean Chebyshev Arithmetic Mean Inequality})
	Let $n \in \mathbb{N}$ and $\mathbb{K}$ be a non-Archimedean valued field.  Let $ r_1, \dots, r_n\in \mathbb{K}$. Let $(a_j)_{j=1}^n, (b_j)_{j=1}^n \in \mathbb{K}^n$. Then 
	\begin{align*}
		\left|\sum_{j=1}^{n}r_ja_j\right|\left|\sum_{k=1}^{n}r_kb_k\right|\leq 	 \max\left\{\left|\sum_{k=1}^{n}r_k\right|\left|\sum_{j=1}^{n}r_ja_jb_j\right|, \max_{1\leq j < k \leq n}|r_j||r_k||a_j-a_k||b_j-b_k|\right\}.
	\end{align*}
	
\end{corollary}
\begin{theorem}(\textbf{Non-Archimedean Chebyshev Arithmetic Mean Inequality})
Let  $\mathbb{K}$ be a non-Archimedean valued field.  Let $ \{r_j\}_{j=1}^\infty\subseteq \mathbb{K}$ be such that $\sum_{j=1}^{\infty}r_j\in \mathbb{K}$. Let $\{a_j\}_{j=1}^\infty, \{b_j\}_{j=1}^\infty  \subseteq \mathbb{K}$. Assume that 
\begin{align*}
	\lim\limits_{j \to \infty }a_j=0 \quad \text{and} \quad 	\lim\limits_{j \to \infty }b_j=0.
\end{align*}
Then 
\begin{align*}
	\left|\sum_{j=1}^{\infty}r_ja_j\right|\left|\sum_{k=1}^{\infty}r_kb_k\right|\leq 	 \max\left\{\left|\sum_{j=1}^{\infty}r_ja_jb_j\right|, \max_{1\leq j < k <\infty }|r_j||r_k||a_j-a_k||b_j-b_k|\right\}.
\end{align*}	
\end{theorem}
\begin{corollary}(\textbf{Non-Archimedean Chebyshev Arithmetic Mean Inequality})
Let  $\mathbb{K}$ be a non-Archimedean valued field.  Let $ \{r_j\}_{j=1}^\infty\subseteq \mathbb{K}$ be such that $\lim\limits_{j \to \infty }r_j=0$. Let $\{a_j\}_{j=1}^\infty, \{b_j\}_{j=1}^\infty  \subseteq \mathbb{K}$. Assume that 
\begin{align*}
	\lim\limits_{j \to \infty }a_j=0 \quad \text{and} \quad 	\lim\limits_{j \to \infty }b_j=0.
\end{align*}
Then 
\begin{align*}
	\left|\sum_{j=1}^{\infty}r_ja_j\right|\left|\sum_{k=1}^{\infty}r_kb_k\right|\leq 	 \max\left\{\left|\sum_{k=1}^{\infty}r_k\right|\left|\sum_{j=1}^{\infty}r_ja_jb_j\right|, \max_{1\leq j < k <\infty }|r_j||r_k||a_j-a_k||b_j-b_k|\right\}.
\end{align*}		
\end{corollary}
\begin{theorem}
(\textbf{Non-Archimedean Chebyshev Arithmetic Mean Inequality})
Let  $\mathbb{K}$ be a non-Archimedean valued field.  Let $ \{r_j\}_{j=1}^\infty\subseteq \mathbb{K}$ be such that $\lim\limits_{j \to \infty }r_j=0$. Let $\{a_j\}_{j=1}^\infty, \{b_j\}_{j=1}^\infty  \subseteq \mathbb{K}$. Assume that 
\begin{align*}
	\lim\limits_{j \to \infty }a_j=0 \quad \text{and} \quad 	\lim\limits_{j \to \infty }b_j=0.
\end{align*}
Then 
\begin{align*}
	\left|\sum_{j=1}^{\infty}r_ja_jb_j\right|\leq \max\left\{\left|\sum_{j=1}^{\infty}r_ja_j\right|\left|\sum_{k=1}^{\infty}r_kb_k\right|, \max_{1\leq j<k<\infty}|r_jr_k(a_j-a_k)(b_j-b_k)|\right\}.
\end{align*}
\end{theorem}
\begin{proof}
	Without loss of generality, we assume that $\sum_{j=1}^{\infty}r_j=1$. Let $n \in \mathbb{N}$. By doing a similar calculation as in the proof of Theorem \ref{NACAMI}, we get 
		\begin{align*}
		\sum_{j=1}^{n}\sum_{k=1}^{n}r_jr_k(a_j-a_k)(b_j-b_k)	&=\sum_{j=1}^{n}\sum_{k=1}^{n}r_jr_k[(a_jb_j-a_jb_k)+(a_kb_k-a_kb_j)]\\
		&=2 \left[\sum_{j=1}^{n}r_ja_jb_j-\left(\sum_{j=1}^{n}r_ja_j\right)\left(\sum_{k=1}^{n}r_kb_k\right)\right].
	\end{align*}
	Rearranging, we get 
	\begin{align*}
	\sum_{j=1}^{n}r_ja_jb_j&=\left(\sum_{j=1}^{n}r_ja_j\right)\left(\sum_{k=1}^{n}r_kb_k\right)+\frac{1}{2}\sum_{j=1}^{n}\sum_{k=1}^{n}r_jr_k(a_j-a_k)(b_j-b_k)\\
	&=	\left(\sum_{j=1}^{n}r_ja_j\right)\left(\sum_{k=1}^{n}r_kb_k\right)+\sum_{1\leq j<k\leq n}r_jr_k(a_j-a_k)(b_j-b_k).	
	\end{align*}
		By taking non-Archimedean absolute value, we get 
			\begin{align*}
			\left|\sum_{j=1}^{n}r_ja_jb_j\right|&=	\left|\left(\sum_{j=1}^{n}r_ja_j\right)\left(\sum_{k=1}^{n}r_kb_k\right)+\sum_{1\leq j<k\leq n}r_jr_k(a_j-a_k)(b_j-b_k)\right|\\
			&\leq \max\left\{\left|\sum_{j=1}^{n}r_ja_j\right|\left|\sum_{k=1}^{n}r_kb_k\right|, \left|\sum_{1\leq j<k\leq n}r_jr_k(a_j-a_k)(b_j-b_k)\right|\right\}\\
			&\leq  \max\left\{\left|\sum_{j=1}^{n}r_ja_j\right|\left|\sum_{k=1}^{n}r_kb_k\right|, \max_{1\leq j<k\leq n}|r_jr_k(a_j-a_k)(b_j-b_k)|\right\}\\
			&\leq \max\left\{\left|\sum_{j=1}^{n}r_ja_j\right|\left|\sum_{k=1}^{n}r_kb_k\right|, \max_{1\leq j<k<\infty}|r_jr_k(a_j-a_k)(b_j-b_k)|\right\}.
		\end{align*}
		By taking limit as  $n \to \infty$, we get the result. 
\end{proof}
\begin{corollary} (\textbf{Non-Archimedean Chebyshev Arithmetic Mean Inequality})
	Let $n \in \mathbb{N}$ and $\mathbb{K}$ be a non-Archimedean valued field. Let $(a_j)_{j=1}^n, (b_j)_{j=1}^n \in \mathbb{K}^n$. Then 
	\begin{align*}
	\left|\sum_{j=1}^{n}a_jb_j\right|\leq \frac{1}{|n|}\max\left\{	\left|\sum_{j=1}^{n}a_j\right|\left|\sum_{k=1}^{n}b_k\right|, \max_{1\leq j < k \leq n}|a_j-a_k||b_j-b_k|\right\}. 
	\end{align*}
	\begin{align*}
		\left|\sum_{j=1}^{n}a_j\right|\left|\sum_{k=1}^{n}b_k\right|\leq 	 \max\left\{|n|\left|\sum_{j=1}^{n}a_jb_j\right|, \max_{1\leq j < k \leq n}|a_j-a_k||b_j-b_k|\right\}.
	\end{align*}
	In other words, 
	\begin{align*}
		\left|AM(a_jb_j)_{j=1}^n\right|\leq \max\left\{	\left|AM(a_j)_{j=1}^n\right|\left|AM(b_j)_{j=1}^n\right|, \frac{1}{|n|^2}\max_{1\leq j < k \leq n}|a_j-a_k||b_j-b_k|\right\}. 
	\end{align*}
\end{corollary}
\section{An Open Problem}
In 1972, Buzano generalized  Cauchy-Schwarz angle-length inequality.
\begin{theorem} \cite{STEELE, FUJIIKUBO} (\textbf{Buzano Angle-Length Inequality}) Let $n \in \mathbb{N}$. Then 
	\begin{align*}
		2\left|\sum_{j=1}^{n}a_jb_j\right|\left|\sum_{k=1}^{n}b_kc_k\right|\leq& \left[ \left(\sum_{j=1}^{n}a_j^2\right)^\frac{1}{2}\left(\sum_{k=1}^{n}c_k^2\right)^\frac{1}{2}+\left|\sum_{j=1}^{n}a_jc_j\right|\right]\sum_{r=1}^{n}b_r^2, \\
		& \forall (a_j)_{j=1}^n, (b_j)_{j=1}^n, (c_j)_{j=1}^n \in \mathbb{R}^n.
	\end{align*}
	\end{theorem}
The Buzano angle-length inequality leads to the following problem.
\begin{problem}
What is the non-Archimedean version of Buzano angle-length inequality (which generalizes non-Archimedean Cauchy-Schwarz angle-length inequality)?
\end{problem}
Note that the Cauchy-Benet identity
\begin{align*}
\left(\sum_{j=1}^{n}a_jc_j\right)\left(\sum_{k=1}^{n}b_kd_k\right)-\left(\sum_{j=1}^{n}a_jd_j\right)\left(\sum_{k=1}^{n}b_kc_k\right)=&\sum_{1\leq j<k \leq n}(a_jb_k-b_ja_k)(c_jd_k-d_jc_k),\\
\forall (a_j)_{j=1}^n, (b_j)_{j=1}^n, (c_j)_{j=1}^n,  (d_j)_{j=1}^n\in \mathbb{K}^n
\end{align*}
 seems not to give non-Archimedean version of Buzano angle-length inequality.

\section{Conclusions}
\begin{enumerate}
	\item In 1821, Cauchy extended Lagrange identity and proved Cauchy-Schwarz inequality \cite{MALIGRANDA}.
	\item In 1882, Chebyshev derived  arithmetic mean inequalities for real numbers  \cite{MITRINOVICVASIC}. 
	\item In this article, we derived non-Archimedean versions of Cauchy-Schwarz angle-length and Chebyshev arithmetic mean inequalities. 
\end{enumerate}

%\section{Data Availability Statement}
%No additional data is generated in the study.

%\section{Disclosure Statement}
%Nothing to disclose.

 \bibliographystyle{plain}
 \bibliography{reference.bib}

\end{document}